\documentclass{rmmcart}
\usepackage{enumitem}
\usepackage{amsmath}
\usepackage{amsfonts}
\usepackage{mathtools}
\usepackage{amssymb}
\usepackage{amsthm}
\usepackage[all,cmtip]{xy}

\newtheorem{theorem}{Theorem}
\newtheorem{lemm}{Lemma}
\newtheorem{defi}{Definition}

\newtheorem{coro}{Corollary}

\newtheorem{ex}{Example}
\newtheorem{rem}{Remark}
\usepackage{fullpage}

\title{Descending chains of semistar operations}

     \author{Hyun Seung Choi}
     \address{Department of Mathematics, Glendale Community College, Glendale, California 91208 U.S.A.}
     \email{hyunc@glendale.edu}
     
     \author{Timothy McEldowney}
     \address{Department of Mathematics, University of California, Riverside, Riverside, California 92521 U.S.A.}
     \email{tmcel001@ucr.edu}

     \author{Andrew Walker}
     \address{Department of Mathematics, Whittier College, Whittier, California 90608 U.S.A.}
     \email{awalker3@whittier.edu}

     \date{November 14, 2017}


\begin{document}

\begin{abstract}
A class of integer-valued functions defined on the set of ideals of an integral domain $R$ is investigated. We show that this class of functions, which we call ideal valuations, are in one-to-one correspondence with countable descending chains of finite type, stable semistar operations with largest element equal to the $e$-operation.  We use this class of functions to recover familiar semistar operations such as the $w$-operation and to give a solution to a conjecture by Chapman and Glaz when the ring is a valuation domain.
\end{abstract}
\maketitle

\newcommand{\Ari}{\mbox{\textup{Ari}}} 
\newcommand{\NDC}{\mbox{\textup{NDC}}}
\newcommand{\Spec}{\mbox{\textup{Spec}}}
\newcommand{\Min}{\mbox{\textup{Min}}} 
\newcommand{\height}{\mbox{\textup{ht}}} 

\section{Introduction}
Semistar operations were first defined in 1994 by Matsuda and Okabe (\cite{Ma2}) as an extension of the classical star operation, which  originated from the work of Krull (\cite{W.Krull.Idealtheorie}) and was formalized by Gilmer (\cite{G}). As a branch of multiplicative ideal theory, star operations have shown to be a capable tool for describing various classes of integral domains and finding new relationships between them. For example, Zafrullah (\cite[Theorem 8]{Z}) proved that an integral domain $R$ is a Pr{\"u}fer domain if and only if $R$ is an integrally closed domain such that the $t$-operation and $d$-operation coincide, and Wang and McCasland (\cite[Theorem 5.4]{WM}) proved that $R$ is a Krull domain if and only if $R$ is a completely integrally closed domain such that the $w$-operation and $v$-operation coincide. In a similar way, semistar operations have received significant interest in the field of multiplicative ideal theory (for instance, see \cite{FH}, \cite{FJS}, \cite{FL}, \cite{Fontana.Picozza}, \cite{Mi}, \cite{Pd}, \cite{P}, \cite{PT}, to name only a few). \\
\indent We investigate here a class of integer-valued functions defined over the set of ideals of a ring, which we call ideal valuations. Their definition is inspired by the properties  shared by many of the classical integer-valued functions defined over the ideals of a ring, such as polynomial grade and height (under mild conditions).  These ideal valuations can also be constructed naturally from localizing systems, which have been studied extensively (see \cite{B}, \cite{FH}, \cite{GJS} and \cite{S}). We first use ideal valuations to establish a link to semistar operations and localizing systems. Specifically, we show that each ideal valuation $\nu$ induces a countable descending chain of localizing systems $G(\nu, n)$ and semistar operations $*_{\nu_{n}}$. We show in Theorem \ref{bijective} in fact that there is a bijection between the set of ideal valuations and the set of countable descending chains of finite type, stable semistar operations with largest operation equal to the $e$-operation. As an application, we recover in Theorem \ref{wfoundit} the $w$-operation as one of the semistar operations in the chain induced by polynomial grade. \\
\indent We also use ideal valuations to characterize  one-dimensional quasi-local domains in Theorem \ref{qlocal} as the domains $R$ where every ideal valuation is constant on proper ideals.   As another application, the last section addresses the question raised by Chapman and Glaz (\cite[Problem 44]{Cha}): If $\{*_{\alpha}\}_{\alpha \in A}$ is a set of semistar operation on $R$, then when is the semistar operation $*_{A}$ defined by $I^{*_{A}}=\bigcap_{\alpha \in A} I^{*_{\alpha}}$ for each $I \in \overline{F}(R)$ of finite type? This question has been considered by Anderson (\cite[Theorem 2]{And}) for star operations, and Mimouni (\cite[Theorem 2.4]{Mi}) for semistar operations over conducive domains. We investigate this question, restricted to the case when $R$ is a valuation domain. In particular, we show that when $\{*_{\alpha}\}_{\alpha \in A}$ form a countable descending chain of stable and finite type semistar operations on a valuation domain, then $*_{A}$ is of finite type precisely when the corresponding ideal valuation has finite range. \\
\indent Throughout, all rings $R$ are assumed to be integral domains. $K$ will denote the quotient field of $R$, $\mathcal{S}(R)$ the set of ideals of $R$, $\mathcal{S'}(R) = \mathcal{S}(R) - \{0,R\}$ the set of proper ideals of $R$, $f\mathcal{S}(R)$ the set of finitely generated ideals of $R$ and $\overline{F}(R)$ the set of nonzero $R$-submodules of $K$. If $I \in \mathcal{S}(R)$, then $V(I)$ will denote the set of prime ideals of $R$ that contain $I$. A quasi-local domain is an integral domain that has only one maximal ideal (not necessarily Noetherian). When we write, $\overline{\mathbb{N}}$, we are referring to the set $\{\infty, 0,1,2,3\ldots,\}$.
\newpage 
\section{Localizing systems and semistar operations}

Recall that a \textbf{localizing system} $\mathcal{F}$ (see \cite[Section 2]{FH}) of an integral domain $R$ is a nonempty family of ideals of $R$ such that the following conditions hold:
\begin{itemize}
\item (LS1) If $I \in \mathcal{F}$ and $J$ is an ideal of $R$ such that $I \subseteq J$, then $J \in \mathcal{F}$. 
\item (LS2) If $I \in \mathcal{F}$ and $J$ is an ideal of $R$ such that $(J :_{R} i) \in \mathcal{F}$ for each $i \in I$, then $J \in \mathcal{F}$.
\end{itemize}
Furthermore, if for every $I \in \mathcal{F}$, there is some $J \subseteq I$ that's finitely generated and $J \in \mathcal{F}$, then we say that $\mathcal{F}$ is of \textbf{finite type}.

\begin{lemm}
\label{1}
Let $R$ be a domain, and $S$ a multiplicatively closed subset of $f\mathcal{S}(R)$. If \[ \mathcal{F}_{S} = \{ I \in \mathcal{S}(R) \mid J \subseteq I \text{ for some }J \in S\},   \] then $\mathcal{F}_{S}$ is a finite type localizing system on $R$. \end{lemm}
\begin{proof}
It's clear that (LS1) holds. We next claim that $\mathcal{F}_{S}$ is multiplicatively closed and is closed under finite intersections. Indeed, if $I,I' \in \mathcal{F}_{S}$, then there are $J,J' \in S$ with $J \subseteq I$ and $J' \subseteq I'$, so that $JJ' \subseteq II' \subseteq I \cap I'$, and since $S$ is multiplicatively closed, $JJ' \in S$. Hence  $II'$ and $I \cap I'$ are both in $\mathcal{F}_{S}$. \\

Now we show that $\mathcal{F}_{S}$ satisfies (LS2). Choose ideals $I,J$ of $R$ so that $I \in \mathcal{F}_{S}$ and $(J:_{R} iR) \in \mathcal{F}_{S}$ for all $i \in I$. We have to show that $J \in \mathcal{F}_{S}$. There exists $I' \subseteq I$ with $I' \in S$. Let $\{i_{k}\}$ be a finite generating set of $I'$. It follows that 
$(J:_{R}I')= (J:_{R} \Sigma i_{k}R)= \cap(J:_{R} i_{k}R)$, so $(J:_{R}I') \in \mathcal{F}_{S}$ and $I'(J:_{R} I')\in \mathcal{F}_{S}$ by the above claim. Then $J \in \mathcal{F}_{S}$ since $ I'(J:_{R} I') \subseteq J$ and by (LS1). Lastly, by definition it follows that $\mathcal{F}_{S}$ is of finite type. \end{proof}

A \textbf{semistar operation} is a map $* \colon \overline{F}(R) \to \overline{F}(R)$ such that for any $I, J \in \overline{F}(R)$ and $x \in K \setminus \{0\}$, 
\begin{enumerate}
\item $I \subseteq I^*$.
\item $I \subseteq J$ implies $I^* \subseteq J^*$.
\item $(xI)^*=xI^*$.
\item $(I^*)^*=I^*$.
\end{enumerate} 

\begin{ex}\normalfont
The following are standard examples of semistar operations:
\begin{itemize}
\item The \emph{identity operation} $d\colon \overline{F}(R) \to \overline{F}(R)$ defined by $I_d=I$ for all $I \in \overline{F}(R)$.
\item The \emph{trivial operation} $e\colon \overline{F}(R) \to \overline{F}(R)$ defined by $I_e=K$ for all $I \in \overline{F}(R)$.
\item The $v$-\emph{operation} $v\colon \overline{F}(R) \to \overline{F}(R)$ defined by $I_v=(R:_{K}(R:_{K}I))$ for all $I \in \overline{F}(R)$.
\item The $t$-\emph{operation} $t\colon \overline{F}(R) \to \overline{F}(R)$ defined by $I_t= \bigcup\{J_v \mid J \subseteq I, J \in f\mathcal{S}(R)\}$ for all $I \in \overline{F}(R)$.
\item The $w$-\emph{operation} $w\colon \overline{F}(R) \to \overline{F}(R)$ defined by $I_w=\bigcup\{(I :_{K} J)\mid J \in f\mathcal{S}(R), J_{v}=R\} $ for all $I \in \overline{F}(R)$.
\end{itemize}
\end{ex}

Recall a semistar operation $* \colon \overline{F}(R) \to \overline{F}(R)$ is said to be \textbf{spectral} if there is some $\Delta \subseteq \text{Spec}(R)$ so that for any $E \in \overline{F}(R)$, \[E^{*} = \bigcap_{P \in \Delta} ER_{P}. \] In this case, we'll write $* = *_{\Delta}$. We say a semistar operation $*$ is \textbf{stable} if $(I \cap J)^*=I^* \cap J^*$ for all $I,J \in \overline{F}(R)$. Given a semistar operation $*$, define $*_{f}$ such that $I^{*_{f}}=\bigcup\{J^* \mid J \subseteq I, J\in f\mathcal{S}(R)\}$ for all $I \in \overline{F}(R)$. Then $*_{f}$ is a semistar operation, and we say $*$ is \textbf{of finite type} if $*=*_{f}$.
Every localizing system $\mathcal{F}$ on a domain $R$ yields a stable semistar operation $*_{\mathcal{F}}$, given as follows (\cite[Proposition 2.4]{FH}): If $I \in \overline{F}(R)$, then \[ I^{*_{\mathcal{F}}} = \bigcup_{J \in \mathcal{F}} (I :_{K} J).      \] 
On the other hand, given a semistar operation $*$ on $R$, the set $\mathcal{F}^*=\{I \in \mathcal{S}(R)\mid I^*=R^*\}$ is a localizing system of $R$ (\cite[Proposition 2.8]{FH}, \cite[Remark 2.9]{FH}). We adopt the notation $\bar{*}$ for the semistar operation $*_{\mathcal{F}^*}$ and $\tilde{*}$ for the semistar operation $\overline{*_{f}}$. That is, $I^{\bar{*}}=\bigcup\{(I :_{K} J)\mid J \in \mathcal{S}(R), J^*=R^*\}$, $I^{\tilde{*}}=\bigcup\{(I :_{K} J)\mid J \in \mathcal{S}(R), J^{*_{f}}=R^*\}$ for all $I \in \overline{F}(R)$ (see \cite[Section 3]{FHP}).  The theorem below gives a relationship between localizing systems of finite type and semistar operations of finite type.

\begin{theorem}
\cite[Proposition 3.2]{FH}
\label{FonHuc}
Let $\mathcal{F}$ be a localizing system and $*$ a semistar operation defined on $R$.
\begin{enumerate}
\item If $\mathcal{F}$ is of finite type, then $*_{\mathcal{F}}$ is of finite type.
\item If $*$ is of finite type, then $\mathcal{F}^*$ is of finite type.
\end{enumerate}
\end{theorem}

\begin{lemm}
Let $R$ be a domain and $S$ a multiplicatively closed subset of $f\mathcal{S}(R)$. Then if $\mathcal{F}_{S}$ is as in Lemma \ref{1}, then for any $I \in \overline{F}(R)$, \[ I^{*_{\mathcal{F}_{S}}} = \bigcup_{J \in \mathcal{F}_{S}} (I :_{K} J) = \bigcup_{L \in S} (I :_{K} L).    \]
\end{lemm}

\begin{proof}
Since $\mathcal{F}_{S}$ is a localizing system, the first equality is by definition of $*_{\mathcal{F}_{S}}$. We now prove the second equality holds. Let $I \in \overline{F}(R)$. Then since $S \subseteq \mathcal{F}_{S}$, obviously \[\cup_{L \in S} (I:L) \subseteq \cup_{J \in \mathcal{F}_{S}}(I:J).\] Conversely, if $J \in \mathcal{F}_{S}$, then there exists $L' \in S$ with $L' \subseteq J$. Thus $(I:J) \subseteq (I:L') \subseteq \cup_{L \in S} (I:L)$, and since this is true for every $J \in \mathcal{F}_{S}$, we have $\cup_{J \in \mathcal{F}_{S}}(I:J) \subseteq \cup_{L \in S} (I:L)$. \end{proof}

\section{Ideal valuation} \begin{defi} \normalfont
Let $R$ be an integral domain. A function $\nu \colon \mathcal{S}(R) \to \overline{\mathbb{N}}$ is an \textbf{ideal valuation} on $R$ if it satisfies the following properties:
\begin{itemize}
\item (IV1) $\nu(0) = 0$ and $\nu(R) = \infty$.
\item (IV2) For any $I,J \in \mathcal{S}(R)$, $\min \{ \nu(I), \nu(J) \} \leq \nu(IJ)$.
\item (IV3) For any $I \in \mathcal{S}(R)$, $\nu(I) = \sup \{ \nu(J) \mid J \subseteq I \text{ and } J \in f\mathcal{S}(R) \}$.
\end{itemize}
\end{defi}

\begin{ex} \normalfont
\label{locex}
If $\mathcal{F}$ is a localizing system of $R$ of finite type, let $\nu_{\mathcal{F}}: \mathcal{S}(R) \rightarrow \overline{\mathbb{N}}$ be the function defined by \[\nu_{\mathcal{F}}(I)=
\begin{cases}
\infty, &\text{ if } I = R \\
1, &\text{if } I \in \mathcal{F} \cap \mathcal{S}'(R)\\
0, &\text{otherwise}. 
\end{cases}
\]
Then $\nu_{\mathcal{F}}$ is an ideal valuation. Indeed, (IV1) is immediate. Given $I,J \in \mathcal{S'}(R)$,  $I, J \in \mathcal{F}$ if and only if $IJ \in \mathcal{F}$. Therefore $\nu_{\mathcal{F}}(IJ) \ge \min\{\nu_{\mathcal{F}}(I), \nu_{\mathcal{F}}(J)\}$, so that (IV2) is met. If $I \not\in \mathcal{F}$, then any ideal contained in $I$ is not an element of $\mathcal{F}$ by (LS1). Hence  \[\nu_{\mathcal{F}}(I)= \sup \{ \nu_{\mathcal{F}}(J) \mid J \subseteq I \text{ and } J \in f\mathcal{S}(R) \},\] so that (IV3) holds. 
\end{ex}

\begin{rem}
\label{cant}
Note that we can't drop the finiteness condition on $\mathcal{F}$ in the previous example. Indeed, let $R$ be a valuation domain whose maximal ideal $M$ is not finitely generated. Then $\{aM\mid a \in R\setminus\{0\}\}$ is the set of nondivisorial ideals of $R$ \cite[Proposition 4.2.5]{FHP}. Thus, in particular, $M$ is nondivisorial and $M_{v}=R$. If $I\neq M$ is a proper ideal of $R$, then $I$ is not principal, so I is not finitely generated and $I_{v}=(aM)_{v}=aM_{v}=aR \neq R$. Hence $\mathcal{F}^{v}=\{I \in \mathcal{S}(R) \mid I_{v}=R\}=\{R, M\}$. Now $\nu_{\mathcal{F}^{v}}(M)=1$, but then $\nu_{\mathcal{F}^{v}}$ is not an ideal valuation since \[\sup \{ \nu_{\mathcal{F}^{v}}(J) \mid J \subseteq M \text{ and } J \in f\mathcal{S}(R) \}=0.\]
\end{rem}

\begin{lemm}
\label{monotonicity}
If $\nu$ is an ideal valuation on $R$, then for  $I,I' \in \mathcal{S}(R)$ and $I \subseteq I'$, $\nu(I) \leq \nu(I')$. 
\end{lemm}

\begin{proof}
Suppose that $J$ is finitely generated and contained in $I$, then by (IV3), $\nu(J) \leq \nu(I')$. Taking the supremum over all such $J$ in $I$, by (IV3) again we have $\nu(I) \leq \nu(I')$.
\end{proof}

\begin{lemm}
\label{equality}
If $\nu$ is an ideal valuation on $R$, then for  $I,J \in \mathcal{S}(R)$, $\nu(I\cap J)=\nu(IJ) =\min \{ \nu(I), \nu(J) \}$. In particular, $\nu(I^n)=\nu(I)$ for all $n \geq 1$ and $I \in \mathcal{S}(R)$.
\end{lemm}

\begin{proof}
By (IV2) and the preceding lemma, $\nu(IJ) \leq \nu(I\cap J)\leq \min \{ \nu(I), \nu(J) \}\leq \nu(IJ)$. The second statement follows by letting $J=I^{n-1}$ and a simple induction.
\end{proof}

\begin{ex} \normalfont If $R$ is a Dedekind domain and $\nu$ an ideal valuation on $R$, then for any $I \in \mathcal{S}'(R)$, $\nu(I)$ is completely determined by the value of $\nu$ on maximal ideals of $R$. In fact, since 
$I=M_{1}M_{2}\cdot\cdot\cdot M_{n} \text{ for some maximal ideals } M_{i} \text{ of } R,$ $\nu(I)= \min_{1\le i\le n}\{\nu(M_{i})\}$.
\end{ex}

\begin{lemm}
Let $R$ be an integral domain and $\nu$ an ideal valuation on $R$. Then given $I \in \mathcal{S}(R)$, $\nu(I)=\nu(\sqrt{I})$.
\end{lemm}

\begin{proof}
By Lemma \ref{monotonicity}, we have $\nu(I)\le \nu(\sqrt{I})$. For the other inequality, note that for each $J\in f\mathcal{S}(R)$ such that $J \subseteq \sqrt{I}$, there exists $n \geq 1$ with $J^{n}\subseteq I$. Thus $\nu(J)=\nu(J^{n})\le \nu(I)$ by Lemma \ref{monotonicity} and Lemma \ref{equality}. Hence $\nu(\sqrt{I})=\sup\{\nu(J)\mid J\subseteq \sqrt{I}, J\in f\mathcal{S}(R)\}\le \nu(I)$.
\end{proof}

\begin{lemm} \label{finite_ext}
If $\nu \colon f\mathcal{S}(R) \to \overline{\mathbb{N}}$ satisfies \normalfont{(IV1)-(IV3)} on $f \mathcal{S}(R)$, then $\nu$ extends uniquely to an ideal valuation $\widetilde{\nu}$ on $R$.
\end{lemm}

\begin{proof}
Given such a $\nu$, define $\widetilde{\nu} \colon \mathcal{S}(R) \to \overline{\mathbb{N}}$ so that $\widetilde{\nu}(I)=\sup \{ \nu(J) \mid J \subseteq I \text{ and } J \in f\mathcal{S}(R) \}$. Then clearly $\widetilde{\nu}$ satisfies (IV1) and (IV3). Now given $I,J \in \mathcal{S}(R)$, suppose that $\widetilde{\nu}(IJ) < \min\{\widetilde{\nu}(I),\widetilde{\nu}(J)\}$. Then by (IV3), there are $I',J' \in f\mathcal{S}(R)$ so that $J' \subseteq J$ and $I' \subseteq I$ and \[  \widetilde{\nu}(IJ) < \widetilde{\nu}(I') \leq \widetilde{\nu}(I) \text{ and }  \widetilde{\nu}(IJ) < \widetilde{\nu}(J') \leq \widetilde{\nu}(J).\]
But then $\widetilde{\nu}(IJ) < \min \{ \widetilde{\nu}(I'),\widetilde{\nu}(J')\} \leq \widetilde{\nu}(I'J') = \nu(I'J') \leq \widetilde{\nu}(IJ)$, a contradiction. Thus $\widetilde{\nu}$ satisfies (IV2). Moreover, uniqueness is immediate from (IV3).
\end{proof}

\begin{coro}
Let $R$ be a domain. If $h \colon \Spec(R) \to \overline{\mathbb{N}}$ satisfies $h(0) = 0$ and $h(P) \leq h(Q)$ whenever $P \subseteq Q$ in $\Spec(R)$, then $h$ induces an ideal valuation $\widetilde{h}$ on $R$. Explicitly, for any $I \in \mathcal{S}(R)$, we have $ \widetilde{h}(I) = \sup \{  \inf \{ h(P) : P \in V(J) \} : J \subseteq I, J \in f\mathcal{S}(R)  \}.  $
\end{coro}

\begin{proof}
We have $h$ defines a map $\overline{h} \colon f\mathcal{S}(R) \to \overline{\mathbb{N}}$, where for $J \in f\mathcal{S}(R)$, \[ \overline{h}(J) = \inf\{ h(P) \mid P \in V(J) \}.     \]

It's clear that since $0$ is a prime ideal, we must have $\overline{h}(0) = 0$. By our convention, the infimum of the empty set is $\infty$, so that $\overline{h}(R) = \infty$. Thus $\overline{h}$ satisfies (IV1) on $f \mathcal{S}(R)$. Now suppose $I,J \in f\mathcal{S}(R)$. Then $h(Q) = \overline{h}(IJ)$ for some $Q \in \Min(IJ)$. Then $Q \in V(I)$ without loss of generality, and so $h(Q) \geq \overline{h}(I)$. Thus (IV2) is met on $\mathcal{S}(R)$. It's clear that (IV3) is satisfied on $f \mathcal{S}(R)$ by the definition of $\overline{h}$ and the monotonicity of $h$, so that by Lemma \ref{finite_ext}, $\overline{h}$ can be extended to an ideal valuation $\widetilde{h}$ on $R$.
\end{proof}
\newpage
\begin{rem}
The ideal valuation $\widetilde{h}$ obtained from $h \colon \Spec(R) \to \overline{\mathbb{N}}$ above is not necessarily an extension of $h$ without an extra condition on the ring $R$, e.g., if $R$ is Noetherian. Indeed, let $V$ be a valuation domain with value group $\prod^{\infty}_{i=1} \mathbb{Z}$ (under lexicographic ordering). Then for each integer $i \geq 0$, $V$ has a unique prime ideal $P_{i}$ of height $i$. The maximal ideal $M$ of $V$ is also the unique prime ideal of infinite height. Therefore, for a prime ideal $P$  properly contained in $M$ there exists prime ideals properly between $P$ and $M$. Thus by \cite[Theorem 17.3.(e)]{G} $M$ is not the radical of a finitely generated (principal) ideal of $V$. Thus, if we let 
\[ h(P) = \left\{
\begin{array}{ll}
      0 & P = 0 \\
      1 & P \notin \{0,M\} \\
      \infty & P = M \\
\end{array} 
\right. \] then $h(M) = \infty$ but $\widetilde{h}(M) = 1$, so that $\widetilde{h}$ is not an extension of $h$. 
\end{rem}

With the above corollary in mind, for an integral domain $R$, we will say that a  function $h \colon \text{Spec}(R) \to \overline{\mathbb{N}}$ is a \textbf{prime valuation} on $R$ if it satisfies the following conditions:
\begin{enumerate}
\item $h(0) = 0$.
\item $h(P) \leq h(Q)$ if $P \subseteq Q$. 
\end{enumerate}
In other words, $h \colon \Spec(R) \to \overline{\mathbb{N}}$ is a prime valuation if and only if $h(0) = 0$ and $h$ is a morphism in the category of posets, where $\Spec(R)$ is partially ordered by inclusion.
\begin{lemm}
Let $R$, $T$ be domains and $R   \to  T$ be a ring homomorphism. If $\nu$ is an ideal valuation on $T$, then $ \nu^{c} \colon \mathcal{S}(R) \to \overline{\mathbb{N}}$ defined by $\nu^{c}(I)=\nu(IT)$ is an ideal valuation on $R$.
\end{lemm}

\begin{proof}
$\nu^{c}$ clearly satisfies (IV1) since $0T =0$ and $RT =T$. Let $I, J \in \mathcal{S}(R)$. Then $\nu^{c}(IJ)=\nu((IJ)T)=\nu(ITJT)\ge \min\{\nu(IT), \nu(JT)\}= \min\{\nu^{c}(I),\nu^{c}(J)\}$, so that (IV2) holds. Lastly, let $I \in \mathcal{S}(R)$ and $L \in f\mathcal{S}(T)$ so that $L \subseteq IT$. Then $L=(\ell_{1},\ldots, \ell_{n})T$ for some $\ell_{k} \in IT$. So $\ell_{k} = \sum i_{kr}t _{kr}$, where $i_{kr} \in I$ and $t_{kr} \in T$. Let $J= \sum i_{kr} R \in f\mathcal{S}(R)$. Then $L \subseteq JT$, so $\nu(L)\le \nu(JT)$ by Lemma \ref{monotonicity}. Since L was arbitrary it follows that $\sup \{ \nu(L) \mid L \subseteq IT \text{ and } L \in f\mathcal{S}(T) \} = \sup \{ \nu(JT) \mid J \subseteq I \text{ and } J \in f\mathcal{S}(R) \}$. Therefore we have $\nu^{c}(I)=\nu(IT)=\sup \{ \nu(L) \mid L \subseteq IT \text{ and } L \in f\mathcal{S}(T)  \} = \sup \{ \nu(JT) \mid J \subseteq I \text{ and }J \in f\mathcal{S}(R) \}=\sup \{\nu^{c}(J) \mid J \subseteq I \text{ and } J \in f\mathcal{S}(R) \}$. Thus $\nu^{c}$ is an ideal valuation of $R$.
\end{proof}

\begin{lemm} Let $R \to T$ be an inclusion of domains. If $\nu$ is an ideal valuation on $R$, then $ \nu^e \colon \mathcal{S}(T)
 \to \overline{\mathbb{N}}$, defined by $\nu^e(I)= \sup\{\nu(J) \mid J \in f\mathcal{S}(R), JT \subseteq I\}$ for each $I \in \mathcal{S}(T)$,  is an ideal valuation on $R$. 
\end{lemm}

\begin{proof}
Clearly $\nu^e$ satisfies (IV1). Suppose $I_{1},I_{2} \in \mathcal{S}(T)$. Assume that both $\nu^e(I_{1})$ and $\nu^e(I_{2})$ are finite. Then there exist $J_{r} \in f\mathcal{S}(R)$ such that $J_{r}T \subseteq I_{r}$ and $\nu^e(I_{r})=\nu(J_{r})$ for $r=1,2$. Now by Lemma \ref{equality} it follows that $\nu^e(I_{1}I_{2}) \ge \nu(J_{1}J_{2})=\min\{ \nu(J_{1}), \nu(J_{2})\}=\min\{\nu^e(I_{1}),\nu^e(I_{2})\}$. Now assume that only one of $\nu^e(I_{1})$ and $\nu^e(I_{2})$ is infinite. Without loss of generality, assume $\nu^e(I_{1})=\infty$. Choose $L \in f\mathcal{S}(R)$ such that $LT \subseteq I_{2}$ and $\nu^e(I_{2})=\nu(L)$. Now there exists $J \in f\mathcal{S}(R)$ such that $\nu(J)>\nu(L)$, $J \in f\mathcal{S}(R)$ and $JT \subseteq I_{1}$. Hence by Lemma \ref{equality} we have $\nu^e(I_{1}I_{2})\ge \nu(JL)=\min\{\nu(J),\nu(L)\}=\nu(L)=\nu^e(I_{2})=\min\{\nu^e(I_{1}),\nu^e(I_{2})\}$. Finally, if $\nu^e(I_{1})=\nu^e(I_{2})=\infty$, then for each $n \geq 1$ there exists $J_{n},L_{n} \in f\mathcal{S}(R)$ such that $J_{n}T \subseteq I_{1}, L_{n}T  \subseteq I_{2}$ with $\nu(J_{n}), \nu(L_{n})>n$. Thus $\nu^e(I_{1}I_{2})\ge \nu(JL)=\min\{\nu(J),\nu(L)\}>n$. Since this is true for arbitrary $n \geq 1$, $\nu^e(I_{1}I_{2})=\infty$. Hence $\nu^e$ satisifes (IV2).

It still remains to show that $\nu^e$ satisifes (IV3).  Suppose that $\nu^e(I)=n < \infty$. Note that for each $I_{1},I_{2} \in \mathcal{S}(T)$, if $I_{1} \subseteq I_{2}$, then $\nu^e(I_{1}) \le \nu^e(I_{2})$. Thus it suffices to show that given $I \in \mathcal{S}(R)$, there exists $J \in fS(R)$ such that  $JT \subseteq I$ and $\nu^e(I) \le \nu^e(JT)$. Now there exists $J \in f\mathcal{S}(R)$ such that $JT \subseteq I$ and $\nu(J)=n$. But then $\nu^e(JT)=\sup\{\nu(L) \mid L \in f\mathcal{S}(R), LT \subseteq JT\} \ge \nu(J)=n$. Hence $\nu^e(I) \le \nu^e(JT)$, and we're done. For the case when $\nu^e(I)= \infty$, given any $n \geq 1$, we have $J_{n} \in f\mathcal{S}(R)$ with $J_{n}T \subseteq I$ such that $\nu(J_{n})\ge n$. Then $\nu^e(J_{n}T)\ge n$, so $\nu^e(I)=\sup\{\nu^e(JT) \mid JT \subseteq I, JT \in f\mathcal{S}(T)\} = \infty$. Therefore $\nu^e$ satisifes (IV3).
\end{proof}

\begin{lemm} Let $R   \to  T$ be an inclusion of domains. Then the following hold: 
\begin{enumerate}[label=(\roman*)]
\item If $\nu$ is an ideal valuation on $R$, then $ \nu^{ece} = \nu^e $. 
\item  If $\nu$ is an ideal valuation on $T$, then $ \nu^{cec} = \nu^c $.
\end{enumerate} 
\end{lemm}

\begin{proof}

$(i)$ Note that given $I \in \mathcal{S}(T)$,
\begin{align*}
\hspace{1.1cm} \nu^{ece}(I) &= \sup\{\nu^{ec}(J) \mid J \in f\mathcal{S}(R), JT \subseteq I\} \\ &= \sup\{\nu^{e}(JT) \mid J \in f\mathcal{S}(R), JT \subseteq I\} & \\ &\le \sup\{\nu^{e}(L) \mid L \in f\mathcal{S}(T), L \subseteq I\} \\ &=\nu^{e}(I), \end{align*} On the other hand, for each $J \in f\mathcal{S}(R)$ we have $\nu(J) \le \nu^{e}(JT)$. So \begin{align*}\nu^{e}(I) &= \sup\{\nu(J) \mid J \in f\mathcal{S}(R), JT \subseteq I\} \\ &\le \sup\{\nu^{e}(JT)\mid J \in f\mathcal{S}(R), JT \subseteq I\}=\nu^{ece}(I), \end{align*} and thus the claim follows.  \\

$(ii)$ Given $I \in \mathcal{S}(R)$, \begin{align*}
\nu^{cec}(I)&=\nu^{ce}(IT)\\ &= \sup\{\nu^{c}(J) \mid J \in f\mathcal{S}(R), JT \subseteq IT\} \\ &=  \sup\{\nu(JT) \mid J \in f\mathcal{S}(R), JT \subseteq IT\}. \end{align*}
 Since $ JT \subseteq IT\ $, we have $\nu(JT)  \le \nu(IT)$ and so \begin{align*} \nu^{cec}(I) &= \sup\{\nu(JT) \mid J \in f\mathcal{S}(R), JT \subseteq IT\} \\ &\le \nu(IT) = \nu^c(I) . \end{align*} On the other hand if $ J' \subseteq IT $ and $ J' \in f\mathcal{S}(T) $ then there is some $ J \in f\mathcal{S}(R)  $  such that $ J' \subseteq JT $.  Thus \begin{align*} \hspace{0.55cm}\nu^c(I) &= \nu (IT)\\ &=  \sup \{ \nu(J') \mid J' \subseteq IT \text{ and } J' \in f\mathcal{S}(T) \}\\  &\le  \sup\{\nu(JT) \mid J \in f\mathcal{S}(R), JT \subseteq IT\}\\  &=   \nu^{cec}(I),\end{align*} and so the claim follows.
\end{proof}

Example \ref{locex} suggests a relationship between localizing systems of finite type and ideal valuations. We'll investigate this further in the following section.

\section{ Relationship between Localizing Systems and Ideal Valuations}

Let $n \geq 0$ and $\nu$ an ideal valuation on a domain $R$. We consider the sets \[ \widetilde{G}(\nu,n) := \{ J \in f\mathcal{S}(R) \mid \nu(J) \geq n\}. \]  It's easy to see that $\widetilde{G}(\nu,n)$ is multiplicatively closed, so that by Lemma \ref{1}, we have an induced localizing system \[G(\nu,n) := \mathcal{F}_{\widetilde{G}(\nu,n)}.\] \newpage In fact, we have that \[ G(\nu,n) = \{ J \in \mathcal{S}(R) \mid \nu(J) \geq n \}, \] which follows immediately from (IV3). So by our work in the previous section, for every ideal valuation $\nu$ and $n \geq 0$ we have a finite type, stable semistar operation $\nu_{n} \colon \overline{F}(R) \to \overline{F}(R)$ given by \[ I \mapsto I_{\nu_{n}}  := \bigcup_{J \in G(\nu,n)} (I :_{K} J) = \bigcup_{J \in \widetilde{G}(\nu,n)} (I :_{K} J).\]

We'll let $\mathcal{C}_{\nu}$ denote this set $\{ \nu_{n} \}^{\infty}_{n=0}$ of stable, finite type semistar operations.

\begin{lemm}
\label{basic arithmetic star} Let $R$ be a domain and $\nu$ an ideal valuation on $R$. Then the following hold:
\begin{enumerate}[label=(\roman*)]
\item If $n \geq m$, then $\nu_{m} \geq \nu_{n}$.
\item $\nu_{0} = e$.
\end{enumerate}
\end{lemm}

\begin{proof}
$(i)$ If $n \geq m$, then $G(\nu,n) \subseteq G(\nu,m)$, so that $\nu_{m} \geq \nu_{n}$. For $(ii)$, just note that $G(\nu,0) = \mathcal{S}(R)$, so that for any $I \in \overline{F}(R)$, $I_{\nu_{0}} = \bigcup_{J \in \mathcal{S}(R)} (I :_{K} J) = K$. 
\end{proof}

By a \textbf{standard, countable descending chain of  finite type, stable semistar operations}, we mean a family $\mathcal{C} = \{ *_{n} \}^{\infty}_{n=0}$  of finite type, stable semistar operations where  $*_{i} \geq *_{i+1}$ for each $i \geq 0$, where $*_{0} = e$. Such a family $\mathcal{C}$ defines now a function  \begin{align*}   \nu_{\mathcal{C}} \colon & \mathcal{S}(R) \to \overline{\mathbb{N}}, \text{ where } \\
& I \mapsto \nu_{\mathcal{C}}(I) = \sup \{ k \mid I_{*_{k}} = R_{*_{k}}  \}  \text{ and } \nu_{\mathcal{C}}(0) = 0. \end{align*} 
\begin{lemm}
Suppose  $\mathcal{C}$ is a standard, countable descending chain of finite type, stable semistar operations on a domain $R$. Then $\nu_{\mathcal{C}}$ is an ideal valuation on $R$. 
\end{lemm}

\begin{proof}
We first observe that since $I_{e} = K = R_{e}$ for any $I \in \mathcal{S}(R)$, it follows that $\nu_{C}(I) \geq 0$. It's clear also that $\nu_{\mathcal{C}}(R) = \infty$, so that (IV1) holds. Suppose now that $I,J \in \mathcal{S}(R)$. We have that $(IJ)_{*_{k}} = (I_{*_{k}}J_{*_{k}})_{*_{k}}$ for all $k \geq 0$. Say $t = \min \{ \nu_{C}(I), \nu_{C}(J)    \}$. Then $(IJ)_{*_{k}} = (I_{*_{k}}J_{*_{k}})_{*_{k}} = (R_{*_{k}}R_{*_{k}})_{*_{k}} = (R_{*_{k}})_{*_{k}} = R_{*_{k}}$ for all $k \leq t$, so that $\nu_{\mathcal{C}}(IJ) \geq t$. Thus (IV2) is met. \\

Say now $I \in \mathcal{S}(R)$ and let $t = \nu_{\mathcal{C}}(I)$. If $t < \infty$, then $I_{*_{t}} = R_{*_{t}}$ but $I_{*_{t+1}} \neq R_{*_{t+1}}$. It follows that $1 \in I_{*_{t}}$, so that since $*_{t}$ is of finite type, there is some $J \subseteq I$ that's finitely generated and $1 \in J_{*_{t}}$. Thus $J_{*_{t}} = R_{*_{t}}$. Even more, $J \subseteq I$ implies $J_{*_{t+1}} \subseteq I_{*_{t+1}} \neq R_{*_{t+1}}$, so that $t = \nu_{\mathcal{C}}(J)$. By the same reasoning, for any $J' \subseteq I$ finitely generated, $\nu_{\mathcal{C}}(J') \leq t$, so that $t = \sup \{ \nu_{\mathcal{C}}(J) \mid J \subseteq I, J \in f\mathcal{S}(R)     \}, $
at least when $t < \infty$. When $t = \infty$,  $1 \in I_{*_{t}}$ for each $t \geq 1$, so that as in the argument above, there is some $J \subseteq I$ finitely generated with $\nu_{\mathcal{C}}(J) = t$. So $\sup \{ \nu_{\mathcal{C}}(J) \mid J \subseteq I, J \in f\mathcal{S}(R)\} = \infty$.  Thus (IV3) holds.
\end{proof}

\begin{theorem}
\label{bijective}
There is a bijective correspondence $\Psi$ from the set of ideal valuations on $R$ to the set of standard, countable descending chain of finite type, stable semistar operations on $R$, given by $ \Psi(\nu) = \mathcal{C}_{\nu}$ with inverse map $\Psi^{-1}$ given by $ \Psi^{-1}(\mathcal{C}) = \nu_{\mathcal{C}}.$    
\end{theorem}

\begin{proof}
Suppose that $\mathcal{C} = \{*_{n}\}^{\infty}_{n=0}$ is a standard, countable descending chain of finite type, stable semistar operations on $R$. We first claim that $\mathcal{C} = \mathcal{C}_{\nu_{\mathcal{C}}}$, or in other words, that $(\nu_{\mathcal{C}})_{n} = *_{n}$ for all $n \geq 0$. Indeed, let $I \in \overline{F}(R)$. Then \[ I_{(\nu_{\mathcal{C}})_{n}} = \bigcup_{J \in G(\nu_{\mathcal{C}},n)} (I :_{K} J).\] Now observe that $G(\nu_{\mathcal{C}},n) = \{ J \in \mathcal{S}(R) \mid \nu_{\mathcal{C}}(J) \geq n  \} = \{ J \in \mathcal{S}(R) \mid J_{*_{n}}  = R_{*_{n}}    \}$. Thus we have that \[ I_{(\nu_{\mathcal{C}})_{n}} = \bigcup_{J_{*_{n}} = R_{*_{n}}} (I :_{K} J) = I_{*_{n}},\] where the second equality holds since $*_{n}$ is a stable semistar operation (\cite[Remark 2.9]{FH} and \cite[Theorem 2.10]{FH}). Thus we've shown that $\Psi(\Psi^{-1}(\mathcal{C})) = \mathcal{C}$. On the other hand, suppose that $\nu$ is an ideal valuation on $R$. We must show that $\nu = \nu_{\mathcal{C}_{\nu}}$. So suppose that $I \in \mathcal{S}(R)$ and that $\nu(I) = t$ for some $t \in \overline{\mathbb{N}}$. Then we have that $I \in G(\nu,k)$ if and only if $k \leq t$, which means that $I_{\nu_{k}} = R_{\nu_{k}}$ if and only if $k \leq t$, so that $\nu_{\mathcal{C}_{\nu}}(I) = t$. Thus $\Psi^{-1}(\Psi(\nu)) = \nu$.
\end{proof}

Let $h$ be a prime valuation on $R$, and for each $n \geq 0$, let 
\[ \Delta_{h_{n}} = \{ P \in \text{Spec}(R) \mid h(P) \leq n  \}.  \] Then we have a family of spectral semistar operations $\mathcal{C}_{h} = \{*_{h_{n}}\}^{\infty}_{n=0}$, where $*_{h_{n}} := *_{\Delta_{h_{n}}}$. Also, $*_{h_{0}} \geq *_{h_{1}} \geq *_{h_{2}} \geq \cdots$, since it's clear that $\Delta_{h_{0}} \subseteq \Delta_{h_{1}} \subseteq \Delta_{h_{2}} \subseteq \ldots$ \\

By a \textbf{countable descending chain of spectral semistar operations}, we mean a  family $\mathcal{C} = \{ *_{k}\}^{\infty}_{k=0}$ of spectral semistar operations where $*_{k} \geq *_{k+1}$ for each $k \geq 0$. From such a family, define a function $h_{\mathcal{C}} \colon \text{Spec}(R) \to \overline{\mathbb{N}}$ by $h_{\mathcal{C}}(0) = 0$, and for $P \in \text{Spec}(R) - \{0\}$, \[ h_{\mathcal{C}}(P) = \inf \{ k \mid P^{*_{k}} \neq R^{*_{k}}   \}.  \] It's easy to see that $h_{\mathcal{C}}$ is a prime valuation on $R$. Indeed, suppose $P \subseteq Q$ and let $t = h_{\mathcal{C}}(Q)$. Then $P^{*_{t}} \subseteq Q^{*_{t}} \neq R^{*_{t}}$, so that $h_{\mathcal{C}}(P) \leq t$.    

We will next establish a bijection between prime valuations and countable descending chains of spectral semistar operations on $R$. First, we require some notation: For a semistar operation $*$ on $R$, let \[  \Pi^{*}=\{P \in \Spec(R)-\{0\} \mid P^{*}\neq R^{*}\}.\]

\begin{lemm}
\label{qu}
\cite[Remark 4.9]{FH}
Let $R$ be a domain, If $*$ is spectral, then $\Pi^{*} \neq \varnothing$ and $I^{*}=\cap\{IR_{P} \mid P \in \Pi^{*}\}$ for each $I \in \overline{F}(R)$.
\end{lemm}

\begin{theorem}
There is a bijective correspondence $\Phi$ from the set of prime valuations on $R$ to the set of countable descending chains of spectral semistar operations on $R$, given by $\Phi(h) = \mathcal{C}_{h}$, with inverse map $\Phi^{-1}(\mathcal{C}) = h_{\mathcal{C}}$.
\end{theorem}

\begin{proof}
Let $h$ be a prime valuation on $R$. We must check that $h = h_{\mathcal{C}_{h}}$. That is, we must show that for all $P \in \text{Spec}(R)$, \[ h(P) = h_{\mathcal{C}_{h}}(P) = \inf \{ k \mid P^{*_{h_{k}}}  \neq R^{*_{h_{k}}}     \}.       \]

Say $t = h(P)$. Now if $k < t$, $P \not\subseteq Q$ for any $Q \in \Delta_{h_{k}}$, since otherwise we would have $t = h(P) \leq h(Q) \leq k$. So, when $k < t$, \[ P^{*_{h_{k}}} = \bigcap_{Q \in \Delta_{h_{k}}} PR_{Q} = \bigcap_{Q \in \Delta_{h_{k}}} R_{Q} = R^{*_{h_{k}}}.      \] Thus $h_{\mathcal{C}_{h}}(P) \geq t$.  If $k = t$, then \[ P^{*_{h_{t}}} = \bigcap_{Q \in \Delta_{h_{t}}} PR_{Q} \subseteq PR_{P},       \] since $P \in \Delta_{h_{t}}$, so that since $1 \notin P^{*_{h_{t}}}$ we have that $P^{*_{h_{t}}} \neq R^{h_{*_{t}}}$, and thus $h_{\mathcal{C}_{h}}(P) = t = h(P)$. So $\Phi^{-1}\Phi(h) = h$ for any prime valuation $h$. \\

Conversely, suppose $\mathcal{C} = \{ *_{k} \}^{\infty}_{k=0}$ is a countable descending of spectral semistar operations. We must show that $\mathcal{C} = \mathcal{C}_{h_{\mathcal{C}}}$. That is, for each $k \geq 0$ we need to show $*_{k} = *_{(h_{\mathcal{C}})_{k}}$. So suppose $E \in \overline{F}(R)$. By Lemma \ref{qu}, we have that \[ E^{*_{k}} = \bigcap_{Q \in \Pi^{*_{k}}} ER_{Q},         \] while \[ E^{*_{(h_{\mathcal{C}})_{k}}} = \bigcap_{P \in \Delta_{(h_{\mathcal{C}})_{k}}} ER_{P} \text{ and } \Delta_{(h_{\mathcal{C}})_{k}} = \{ P \in \text{Spec}(R) \mid h_{\mathcal{C}}(P) \leq k \}.     \]

If $P \in \Delta_{(h_{\mathcal{C}})_{k}}$, then $h_{\mathcal{C}}(P) \leq k$. So $P^{*_{t}} \neq R^{*_{t}}$ for some $t \leq k$. Since $*_{t} \geq *_{k}$, we have that $P^{*_{k}} \neq R^{*_{k}}$. Hence $P \in \Pi^{*_{k}}$, and we can conclude that $\Delta_{(h_{\mathcal{C}})_{k}} \subseteq \Pi^{*_{k}}$.  On the other hand, if $P \in  \Pi^{*_{k}}$ and $P \neq 0$, then $P^{*_{k}} \neq R^{*_{k}}$, so $h_{\mathcal{C}}(P) \leq k$ and $P \in \Delta_{(h_{\mathcal{C}})_{k}}$. Thus $\Delta_{(h_{\mathcal{C}})_{k}}=\Pi^{*_{k}} \cup \{0\}$, and since $ER_{(0)} = K$ for any $E \in \overline{F}(R)$, we have $*_{k}=*_{(h_{\mathcal{C}})_{k}}$.
We conclude that for all $E \in \overline{F}(R)$, $E^{*_{k}} = E^{*_{(h_{\mathcal{C}})_{k}}}$, and thus $\Phi\Phi^{-1}(\mathcal{C}) = \mathcal{C}$.
\end{proof}

We'll say that an ideal valuation $\nu$ is \textbf{constant on proper ideals} of $R$ if there is some $c \in \overline{\mathbb{N}}$ so that  for any $I \in \mathcal{S}'(R)$, $\nu(I) = c$. Similarly, say that a prime valuation $h$ is \textbf{constant on nonzero prime ideals} of $R$ if there is some $d \in \overline{\mathbb{N}}$ so that for any nonzero prime ideal $P$ of $R$, $h(P) = d$.

\begin{theorem}{\label{qlocal}}
Let $R$ be an integral domain that is not a field. Then the following are equivalent.
\begin{enumerate}
\item $R$ is a one-dimensional quasi-local domain.
\item $\tilde{*}=d$ for every semistar operation $* \neq e$ on $R$.
\item If $*$ is a semistar operation on $R$ that is both stable and of finite type, then either $*=e$ or $*=d$.
\item Every ideal valuation on $R$ is constant. 
\item Every prime valuation on $R$ is constant and at least one ideal valuation takes a nonzero value.
\end{enumerate}
\end{theorem}

\begin{proof}
$(1) \Leftrightarrow (2)$: See \cite[Theorem 2.8]{PT}. $(2) \Leftrightarrow (3)$: This follows from the fact that $*=\tilde{*}$ if and only if $*$ is stable and of finite type (\cite[Corollary 3.9]{FH}).\\
\\
$(3) \Rightarrow (4)$: Suppose that $d$ and $e$ are the only stable and of finite type semistar operations of $R$. Then by Lemma \ref{basic arithmetic star}, given an ideal valuation $\nu$ of $R$, either $\nu_{n}=e$ for all $n \geq 0$, or there exists $r \geq 0$ such that $\nu_{n}=e$ for all $0 \le n\le r$ and $\nu_{n}=d$ for all $n>r$. Now suppose that $\nu_{n}=e$ for all $n \geq 0$. Then $\nu(I)=\infty$ for all $I \in \mathcal{S}(R) \setminus \{0\}$. Indeed, by Theorem \ref{bijective}, 
$G(\nu, n) = \{ J \in \mathcal{S}(R) \mid \nu_{n}(J) \geq n  \} = \{ J \in \mathcal{S}(R) \mid J_{v_{n}}  = R_{v_{n}}    \}=\mathcal{S}(R)$,
 so $\nu(I) \ge n$ for any $I \in \mathcal{S}'(R)$. Since this inequality holds for arbitrary $n$, $\nu(I)=\infty$. Thus $\nu$ is constant on proper ideals. On the other hand, suppose that $\nu_{r}=d$ for some $r \geq 1$ with $r$ chosen minimally. Now by the above argument, $\nu(I)\ge r$ for all $I \in \mathcal{S}'(R)$. Assume that $\nu(I)>r$. Then $I \in G(\nu,r+1)$, so that $I_{\nu_{r+1}}=R_{\nu_{r+1}}$, but since $\nu_{r+1}=d$ by assumption, $I=R$. Thus $\nu(I)=r$ for all $I \in \mathcal{S'}(R)$, meaning  $\nu$ is constant on proper ideals.\\
\\
$(4) \Rightarrow (2)$: Suppose that every ideal valuation on $R$ is constant on proper ideals. Let $* \neq e$ be a semistar operation on $R$. Then $\mathcal{F}^{*_{f}}=\{I \in \mathcal{S}(R) \mid I^{*_{f}}=R^{*}\}$ is a localizing system of finite type by Theorem \ref{FonHuc}. Set $\nu=\nu_{\mathcal{F}^{*_{f}}}$, as defined in Example \ref{locex}. Now choose nonzero $x \in R$ such that $x$ is not a unit of $R^*$. Then $(xR)^{*_{f}}=xR^{*_{f}}\neq R^{*}$, so $\nu(xR)=0$ and thus $\nu(I)=0$ for all $I \in \mathcal{S'}(R)$ by assumption. So $\mathcal{F}^{*_{f}}=\{R\}$ and $I^{\tilde{*}}=I_{\mathcal{F}^{*_{f}}}=I:_{K}R=I$ for all $I \in \overline{F}(R)$, and hence  $\tilde{*}=d$. \\ 
\\
 $(1) \Rightarrow (5)$: This is immediate, since in this case $\text{Spec}(R)- \{0\}$ consists of a single element, so any prime valuation function $h$ is automatically constant on nonzero prime ideals of $R$.  
 
 \newpage $(5) \Rightarrow (1)$: Suppose that every prime valuation on $R$ is constant on nonzero prime ideals of $R$. Fix $Q \in \text{Spec}(R) - \{0\}$, and define $h \colon \text{Spec}(R) \to \overline{\mathbb{N}}$ by \[ h(P) = 
  \begin{cases}
  0 &\text{ if } Q = 0 \\
  \infty &\text{ if }  P \in V(Q) \\
  1 &\text{otherwise}
  \end{cases}
  \] It's easy to see that $h$ is a prime valuation, so that we must have $V(Q) = \Spec(R) - \{0\}$. Since this holds for any nonzero prime ideal $Q$, we must have that $|\Spec(R)| = 2$, meaning $R$ is $1$-dimensional quasi-local.
 \end{proof}

\section{Polynomial grade and height}

Recall that we say that a sequence $\textbf{a} = a_{1},\ldots, a_{r}$ of elements of an arbitrary ring $R$ is a 
\textbf{weak R-sequence} if for each $i = 1,\ldots, r$,  $a_{i} \notin Z (R/(a_{1},\ldots, a_{i-1}))$, where we write $Z(M)$ to denote the set of zerodivisors of an $R$-module $M$. 
 For an ideal $I$ of $R$, we let $\text{grade}(I)$ denote the supremum over 
all lengths of weak $R$-sequences that are contained in $I$. This is also called the 
\textbf{classical grade} of $I$ on $R$. In the non-Noetherian setting, classical grade behaves strangely: there are ideals $I$ that consist of zerodivisors, yet $(0 :_{R} I) = 0$. Passing to the polynomial ring $R[x]$ can resolve this issue, laying outing out the following notion: The \textbf{polynomial grade} of an ideal $I$ of $R$, denoted by $\text{p.grade}(I)$, 
is the value \[ \text{p.grade}(I) = \lim_{m \to \infty} \text{grade}_{R[X_{1},\ldots, X_{m}]}\Big(I[X_{1},\ldots,X_{m}]\Big),\] where $R[X_{1},\ldots, X_{m}]$ denotes the polynomial ring in $m$ variables over $R$.

\begin{lemm}
\label{pgradeproperties}
Let $I$ be an ideal of an arbitrary ring $R$. Then the following hold:
\begin{enumerate}
\item $\text{grade}_{R}(I) \leq \text{p.grade}_{R}(I)$. If $R$ is Noetherian and $I$ is proper, then we have equality.
\item If $I \neq R$, then there is a prime $\mathfrak{p} \in V(I)$ so that $\text{p.grade}(\mathfrak{p}) = \text{p.grade}(I).$ Moreover, $\text{p.grade}(I) = \text{p.grade}(\sqrt{I})$.
\item Let $M$ be an $R$-module and $f \in R[X_{1},\ldots, X_{m}]$ and suppose $I$ is generated by the coefficients of $f$. Then $0 :_{M} I = 0 \Leftrightarrow$ $f$ is a non-zerodivisor on $M[X_{1},\ldots,X_{n}]$.
\item If $J$ is another ideal with $J \subseteq I$, then $\text{p.grade}(J) \leq \text{p.grade}(I)$.
\item $\text{p.grade}(I) = \sup \{ \text{p.grade}(J) \mid J \subseteq I, J \in f\mathcal{S}(R) \}$.
\end{enumerate}
\end{lemm}
\begin{proof}
These statements, along with their proofs, can all be found in \cite[Chapter 5]{Nor}.
\end{proof}

\begin{coro}
Let $R$ be a domain. Then $\text{p.grade} \colon \mathcal{S}(R) \to \overline{\mathbb{N}}$ is an ideal valuation on $R$.
\end{coro}

\begin{proof}
(IV1) and (IV3) follow immediately from the above lemma. Now if $I,J \in \mathcal{S}(R)$ and $t = \text{p.grade}(IJ)$, then there is some prime ideal $\mathfrak{p}$ of $R$ so that $\mathfrak{p} \in V(IJ) = V(I) \cup V(J)$ and $\text{p.grade}(\mathfrak{p})  = t$. Thus $\text{p.grade}(I) \leq t$ or $\text{p.grade}(J) \leq t$, and so (IV2) holds.
\end{proof}

Thus by Theorem \ref{bijective}, the ideal valuation p.grade yields a family $\{ (\text{p.grade})_{n} \}^{\infty}_{n=0}$ of finite type semistar operations on $R$. We next relate the $(\text{p.grade})_{n}$ operations to some familiar semistar operations for small values of $n$. First, we need a preparatory lemma.

\begin{lemm}Let $J$ be a finitely generated ideal in a domain $R$. Then $J^{-1} = R$ if and only if $\text{p.grade}(J) \geq 2.$
\end{lemm}
\begin{proof}
If $J = R$, the claim is trivial, so we can assume throughout that $J$ is proper. So if $J^{-1} = R$, then since $J$ is proper, it's not principal. Thus we may write $J = Rb + I$ for some 
non-zero $b \in J$, where $I \not\subseteq Rb$ and $I$ is finitely generated. So write $I = (a_{0},\ldots, a_{n})$,
and let $f = a_{0} + a_{1}X + \ldots + a_{n}X^{n} \in R[X]$. We claim that $b,f$ is a $R[X]$-sequence contained in $J[X]$. Clearly $b$ is $R[X]$-regular. Next, we claim that $f$ is a non-zerodivisor on $R[X]/bR[X] \cong (R/bR)[X]$. By Lemma~\ref{pgradeproperties}(3), this happens precisely when $(0 :_{R/bR} I) = 0$. So suppose $r + bR \in (0 :_{R/bR} I)$. Then $Ir + bR = bR$, so
$Ir \subset bR \Rightarrow I(r/b) \subset R$. Then $r/b \in J^{-1} = R$, or in other words, $r \in bR$. Hence we have 
$(0 :_{R/bR} I) = 0$, so that $b,f$ forms a $R[X]$-sequence in $J[X]$. Thus $\text{p.grade}(J) \geq 2$. Conversely, suppose $\text{p.grade}(J) \geq 2$, so that $\text{grade}_{S}(JS) \geq 2,$ where $S$ is a polynomial ring in finitely many variables over $R$.
If $\text{grade}_{S}(JS) \geq 2$, then $JS$ contains an $S$-sequence $f,g$. Now if 
$c \in (JS)^{-1}$, then $cf = t$ and $cg = s$ for some $s,t \in S$. Hence we have $sf = cfg = tg$, which means $t = uf$ for some $u \in S$ since $g$ is a non-zerodivisor on $R/fR$. Thus $cf = t = uf$, which means $c = u \in R$. So $(JS)^{-1} = S,$ and hence  $J^{-1} = R$. 
\end{proof}

\begin{theorem}\label{wfoundit}
Let $R$ be a domain. Then for $i \in \{0,1,2\}$ we have,  
\[ (\text{p.grade})_{i} = \begin{cases}
e, &\text{if } i=0,1.\\
w, &\text{if } i = 2. 
\end{cases}\]
\end{theorem}

\begin{proof}
First, since $\text{p.grade}$ is an ideal valuation, we have $(\text{p.grade})_{0} = e$. Now recall that \[ G(\text{p.grade},1) = \{ J \in \mathcal{S}(R) \mid \text{p.grade}(J) \geq 1  \}. \] On the other hand, since $R$ is a domain, every nonzero ideal $J$ has $\text{p.grade}(J) \geq 1$, so that $G(\text{p.grade},1) = \mathcal{S}(R) - \{0\}$, and thus \[ I_{(\text{p.grade})_{1}} = \bigcup_{J \in \mathcal{S}(R) - \{0\}} (I :_{K} J) = K \] for any $I \in \overline{F}(R)$. Thus $(\text{p.grade})_{1} = e$. Lastly, to see $(\text{p.grade})_{2} = w$, observe by the above lemma we have that \begin{align*} \widetilde{G}(\text{p.grade},2)  &= \{ J \in f\mathcal{S}(R) \mid \text{p.grade}(J) \geq 2 \} \\ &= \{ J \in f\mathcal{S}(R) \mid J^{-1} = R  \} \\ &= \{ J \in f\mathcal{S}(R) \mid J_{v} = R  \}.   \end{align*} Thus for any $I \in \overline{F}(R)$, we have that \[ I_{(\text{p.grade})_{2}} = \bigcup_{J \in \widetilde{G}(\text{p.grade},2)} (I _{:K} J) = \bigcup_{J \in f\mathcal{S}(R), J_{v} = R} (I :_{K} J ) = I_{w}.\] Thus $(\text{p.grade})_{2} = w$.
\end{proof}

%
%
%
%
%

\begin{ex} \normalfont
Let $R = k[X,Y,Z]$ with $k$ a field and consider the ideal $I = (X^{2}Z, Z^{2}Y, Y^{2}X)$. As discussed above, $\text{p.grade}$ is an ideal valuation on $R$. Now, we have a primary decomposition of $I$:
\[ I = (X^{2},Y)\cap (X,Z^{2}) \cap (Y^{2},Z)\cap(X^{2},Y^{2},Z^{2}).          \]
Since $R$ is Noetherian, grade and $\text{p.grade}$ coincide. Then the primary components have grades $2,2,2,3$ respectively. Thus we have for any $n \geq 1$,

\[  I_{(\text{p.grade})_{n}} = \left\{
\begin{array}{ll}
      K & \text{if } n = 0,1  \\
      R & \text{if } n = 2 \\
      (X^{2},Y)\cap (X,Z^{2}) \cap (Y^{2},Z) = (YZ^{2},XYZ,X^{2}Z,XY^{2}) & \text{if } n = 3 \\
      I & \text{if } n > 3.
\end{array} 
\right. \]
\end{ex}

The height function on ideals almost determines an ideal valuation $\text{ht} \colon \mathcal{S}(R) \to \overline{\mathbb{N}}$. For instance, $\text{ht}(I) = \text{ht}(\sqrt{I})$ for any $I \in \mathcal{S}(R)$, and $\height$ satisfies (IV1) and (IV2). In general though, (IV3) may not be satisfied by the height function.
\newpage 
\begin{ex} \normalfont
Let $E$ be the set of entire functions on the complex plane $\mathbb{C}$. The following are well known properties of $E$ (see \cite[Exercises 13.16-20]{G}):
\begin{enumerate}[label=(\roman*)]
\item $E$ is a non-Noetherian B{\'e}zout domain.
\item  Given $\alpha \in \mathbb{C}$, $p_{\alpha}E$ is a maximal ideal of $E$ that has height 1, where $p_{\alpha} \colon \mathbb{C} \to \mathbb{C}$ is the entire function that sends $z \mapsto z - \alpha$.
\item $\bigcup_{\alpha \in \mathbb{C}} p_{\alpha}E $ is the set of nonunits of $E$.
\item $E$ has a prime ideal $P$ that has infinite height.
\end{enumerate}
Now given any finitely generated ideal $J \subseteq P$ of $E$, $J=aE$ for some $a \in E$ since $E$ is a B{\'e}zout domain. Since $a$ is a nonunit, $a \in p_{\alpha}E$ for some $\alpha \in \mathbb{C}$. Hence $\height(J)=\height(aE)= \inf\{ \height(M)\mid M \in V(aE)\} \le \height (p_{\alpha}E)=1$, and so  $\sup\{\height(J) \mid J \subseteq P, J \in f\mathcal{S}(E)\}\le 1<\infty= \height(P)$. Therefore the height function fails to satisfy (IV3).
\end{ex}

We will introduce next a fairly weak condition on the ring so that $\height \colon \mathcal{S}(R) \to \overline{ \mathbb{N}}$ will be an ideal valuation on $R$. A ring $R$ is said to be \textbf{FGFC} (see \cite{Mar}) if each finitely generated ideal of $R$ has only finitely many minimal primes. We first observe the following property of FGFC rings:

\begin{lemm}
Let $R$ be a FGFC ring and $I$ an ideal with $\text{ht}(I) \geq k$. Then there are $x_{1},\ldots, x_{k} \in I$ such that $\height((x_{1},\ldots,x_{i})R) \geq i$ for each $i \in \{1,\ldots, k\}$.
\end{lemm}

\begin{proof}
If $k = 0$, we have nothing to do, so we proceed by induction on the height of $I$. So suppose that $\text{ht}(I) = k > 0$. Now if for every $x \in I$, there is some minimal prime $\mathfrak{p}$ such that $x \in \mathfrak{p}$, then since $R$ is FGFC and prime avoidance, $I$ is contained in some minimal prime, a contradiction. Thus there is some $x_{1} \in I$ such that $x_{1}$ does not lie in any minimal prime $\mathfrak{p}$ of $R$. So $\text{ht}(x_{1}R) \geq 1$. 

 Now $Rx_{1}$ is finitely generated, so that by \cite[Proposition $2.2$(b)]{Mar} $R' := R/x_{1}R$ is also a FGFC ring and for $I' := IR'$, we must have that $\text{ht}(I') \geq k -1$.  Indeed, if not, suppose that $\text{ht}(I') < k -1 $. Then there is a chain of prime ideals of length $t < k - 1$ such that the top prime ideal contains $I'$: $P'_{0} \supset \ldots \supset P'_{t}$. This lifts to a chain of prime ideals in $R$ that contain $Rx_{1}$: $P_{0} \supset P_{1} \ldots \supset P_{t} \supset Rx_{1}$. Now $P_{t}$ cannot be a minimal prime of $R$, hence this chain be extended by one. Since $P_{0} \subseteq I$, we get $\height(I) < k$, which is a contradiction. Thus by induction, we may choose elements $x_{2},\ldots, x_{n} \in I$ such that for their images $x'_{2},\ldots, x'_{n} \in R'$, $\height((x'_{2},\ldots,x'_{i})R') \geq i - 1$ for $i = 2,\ldots,k$. It then follows by a similar argument as before that $(x_{1},\ldots,x_{i})R$ has height bigger than or equal to $i$.  
\end{proof}

\begin{coro} Let $R$ be a FGFC domain. Then $\height \colon \mathcal{S}(R) \to \overline{\mathbb{N}}$ is an ideal valuation on $R$. \end{coro}
\begin{proof}
As stated earlier, all we must check is (IV3), which follows immediately by the above lemma.
\end{proof}

\section{A question of Chapman and Glaz}

It is well-known that given a collection of overrings $\{R_{\alpha}\}_{\alpha \in A}$ of a domain $R$, if $*_{\alpha}$ is a semistar operation on $R_{\alpha}$ for each $\alpha \in A$, then the map $*_{A}\colon \overline{F}(R) \to \overline{F}(R)$ defined by $I^{*_{A}}=\cap_{\alpha \in A}(IR_{\alpha})^{*_{\alpha}}$ for each $I \in \overline{F}(R)$ is a semistar operation (\cite[Example 1.3 (d)]{FH}).
In this section, we wish to investigate the following problem posed in \cite[Problem 44]{Cha}: Find conditions for $*_{A}$ to be of finite type, or equivalently, if $\{*_{\alpha}\}_{\alpha \in A}$ is a set of semistar operation on $R$, then when is the semistar operation $*_{A}$ defined by $I^{*_{A}}=\bigcap_{\alpha \in A} I^{*_{\alpha}}$ for each $I \in \overline{F}(R)$ of finite type? 

We will consider this question under the assumption that $R$ is a valuation domain, or in other words,  when $\overline{F}(R)$ is totally ordered under inclusion. In this scenario, every semistar operation is stable. Indeed, given $I,J \in \overline{F}(R)$ in a valuation domain $R$, $I \subseteq J$ without loss of generality, so $(I \cap J)^{*}=I^{*}=I^{*}\cap J^{*}$ for each semistar operation $*$ on $R$. From this observation, we have the following lemma:
\newpage
\begin{lemm}
\label{y2}
Let $R$ be a domain and $\{*_{\alpha}\}_{\alpha \in A}$ a set of semistar operations on $R$.  Then the following hold:
\begin{enumerate}
\item
$\mathcal{F}^{*_{A}}=\cap_{\alpha \in A}\mathcal{F}^{*_{\alpha}}$.
\item
If every semistar operation on $R$ is stable, then $*_{A}$ is of finite type if and only if $\cap_{\alpha \in A}\mathcal{F^{*_{\alpha}}}$ is of finite type.  
\end{enumerate}
\end{lemm}
\begin{proof}
$(1)$: Note that $(I^{*{_{A}}})^{*_{\beta}} = I^{*_{\beta}}$ for each $I \in \overline{F}(R)$ and $\beta \in A$ by  (\cite[Proposition 1.6 (4)]{FH}). Therefore, 
\begin{align*}
 I \in \mathcal{F}^{*_{A}}
&\Leftrightarrow I^{*_{A}}=R^{*_{A}} \\
&\Leftrightarrow I^{*_{\alpha}}=R^{*_{\alpha}} \text{for each } \alpha \in A\\ 
&\Leftrightarrow I \in \mathcal{F}^{*_{\alpha}} \text{for each } \alpha \in A\\
&\Leftrightarrow I \in \cap_{\alpha \in A}\mathcal{F}^{*_{\alpha}}.
\end{align*}
$(2)$: Given a semistar operation $*$ on $R$, $*$ is of finite type if and only if $\mathcal{F}^{*}$ is of finite type, by our  assumption, along with Theorem \ref{FonHuc} and \cite[Theorem 2.10(A)]{FH}. Now by $(1)$, $*_{A}$ is of finite type if and only if $\mathcal{F}^{*_{A}}$ is of finite type if and only if $\cap_{\alpha \in A}\mathcal{F^{*_{\alpha}}}$ is of finite type.  
\end{proof}
 
\begin{lemm}
\label{f}
Let $R$ be a valuation domain. Then the following statements hold:
\begin{enumerate} \item (\cite[Lemma 2.40]{Pd}) If $*$ is a finite type semistar operation on $R$, then there is an overring $T$ of $R$ so that $I^{*} = IT$ for each $I \in \overline{F}(R)$. \item (\cite[Theorem 10.1]{Mat}) Each overring of $R$ is of the form $R_{P}$, where $P \in \Spec(R)$.
\end{enumerate}
\end{lemm}

\begin{coro}
\label{r}
Let $R$ be a valuation domain and $*$ a semistar operation on $R$. Then the following are equivalent:
\begin{enumerate}
\item  $*$ is of finite type.
\item For some $P \in \Spec(R)$,  $I^{*}=IR_{P}$ for each $I \in \overline{F}(R)$. 
\item There is an overring $T$ of $R$ so that $I^{*}=IT$ for each $I \in \overline{F}(R)$.
\end{enumerate}
\end{coro}

From this, we obtain a characterization of localizing systems of finite type on a valuation domain.

\begin{lemm}
\label{hy}
Let $R$ be a valuation domain and $\mathcal{F}$ a localizing system of $R$. Then $\mathcal{F}$ is of finite type if and only if $\mathcal{F}=\{I \in \mathcal{S}(R)\mid P\subsetneq I\}$ for some prime ideal $P$ of $R$.
\end{lemm}

\begin{proof}
Suppose that $\mathcal{F}$ is of finite type. Note first that $*_{\mathcal{F}}$ is a semistar operation of finite type on $R$ by Theorem \ref{FonHuc}. Thus by Corollary \ref{r}, there is $P \in \Spec(R)$ with $I^{*_{\mathcal{F}}}=IR_{P}$ for each $I \in \overline{F}(R)$. Now, by \cite[Theorem 2.10(A)]{FH}, we have
\begin{align*} \mathcal{F}
&=\{I\in \mathcal{S}(R)\mid I^{*_{\mathcal{F}}}=R^{*_{\mathcal{F}}}\}\\
&=\{I\in \mathcal{S}(R)\mid IR_{P}=R_{P}\}\\
&=\{I\in \mathcal{S}(R)\mid I\not\subseteq P\}\\
&=\{I\in \mathcal{S}(R)\mid P\subsetneq I\}.
\end{align*}
Conversely, say $\mathcal{F}=\{I \in \mathcal{S}(R)\mid Q\subsetneq I\}$ for some $Q \in \Spec(R)$. Then given $I \in \mathcal{F}$, choose $x\in I\setminus Q$. Since $R$ is a valuation domain, $Q\subsetneq xR\subseteq I$ and $xR \in \mathcal{F}$. Therefore $\mathcal{F}$ is of finite type.
\end{proof}


%
%
\newpage
\begin{lemm}
Let $R$ be a valuation domain, and $\nu$ an ideal valuation of $R$. Then \[\dim(R) +1  \geq \#\{\nu(I)\mid I\in \mathcal{S}(R)\}.\]
\end{lemm}

\begin{proof}
For each $n > 0$, let $P_{n} = \{ x \in R \mid \nu(xR) < n \}$. First, we show that $P_{n}$ is an ideal of $R$.  Indeed, let $x,y \in P_{n}$ and say $r \in R$. Without loss of generality, we have $xR \subseteq yR$, so that by monotonicity of $\nu$, it follows that $\nu( (x+y)R) \leq \nu(xR + yR) = \nu (yR) < n$. Hence $x+y \in P_{n}$. Also by monotonicity of $\nu$, we have $\nu( (rx)R) \leq \nu(xR) < n$, and so $rx  \in P_{n}$. Thus each $P_{n}$ is an ideal of $R$. Now suppose $I,J$ are ideals of $R$ with $IJ \subseteq P_{n}$. Then either $I \subseteq P_{n}$ or $J \subseteq P_{n}$ by (IV2). Thus for each $n > 0$ we have $P_{n} \in \Spec(R)$, and since the $\{ P_{n} \}^{\infty}_{n=1}$ form an ascending chain, the claim follows.
\end{proof}

\begin{lemm}
\label{y1}
If $\nu$ is an ideal valuation on a valuation domain $R$, then $\nu$ has finite range if and only if $\cap_{n\ge0}G(\nu,n)$ is a localizing system of $R$ of finite type.
\end{lemm}

\begin{proof}
If $\cap_{n\ge0}G(\nu,n)$ is a localizing system of $R$ of finite type, then by Lemma \ref{hy}, there is $P \in \Spec(R)$ so that $\nu(I)=\infty$ if and only if $P\subsetneq I$. In particular, $\nu(P)=m$ for some $m \geq 0$. So, for any $I \in \mathcal{S}(R)$, either $\nu(I)\le m$ or $\nu(I)=\infty$. Therefore $\nu$ has finite range.  On the other hand, if the range of $\nu$ is finite, then there is $m \geq 0$ so that $G(\nu, m)=\cap_{n\ge0}G(\nu,n)$.
\end{proof}


\begin{theorem}
\label{vs}
Let $R$ be a valuation domain, $\{*_{\alpha}\}_{\alpha \in A}$ a set of semistar operations on $R$. 
Then the following are equivalent.
\begin{enumerate}
\item
$*_{A}$ is of finite type.
\item
$\cap_{\alpha \in A}\mathcal{F^{*_{\alpha}}}$ is of finite type.
\item
There exists $\alpha \in A$ such that $I^{*_{A}}=IR^{*_{\alpha}}$ for each $I \in \overline{F}(R)$. 
\end{enumerate}
\end{theorem}

\begin{proof}
$(1) \Leftrightarrow (2)$: This follows from Lemma \ref{y2}. $(1) \Rightarrow (3)$: Note that for each $\alpha \in A$, there is some $P_{\alpha} \in \Spec(R)$ such that $R^{*_{\alpha}}=R_{P_{\alpha}}$ by Lemma \ref{f}. 
Suppose that $*_{A}$ is of finite type. Then by Corollary \ref{r}, there is $P \in \Spec(R)$ with  $I^{*_{A}}=IR_{P}$ for each $I \in \overline{F}(R)$. It follows  $P_{\alpha}\subseteq P$ for each $\alpha \in A$. If $P_{\alpha}\subsetneq P$ for each $\alpha \in A$, then $P^{*_{A}}=\cap_{\alpha \in A} P^{*_{\alpha}} \supseteq \cap_{\alpha \in A}PR_{P_{\alpha}}=\cap R_{P_{\alpha}}=R^{*_{A}}=R_{P} \supsetneq PR_{P}=P^{*_{A}}$, a contradiction. Hence $P=P_{\alpha}$ for some $\alpha \in A$, and thus $R^{*_{\alpha}}=R_{P_{\alpha}}=R_{P}$. So $I^{*_{A}}=IR^{*_{\alpha}}$ for each $I \in \overline{F}(R)$. $(3) \Rightarrow (1)$: This follows from Corollary \ref{r}.
\end{proof}

When we have a standard, countable descending chain of finite type and stable semistar operations over a valuation domain, we can say a little more:

\begin{coro}
Let $R$ be a valuation domain and $\mathcal{C}=\{*_{i}\}^{\infty}_{i = 0}$ a standard, countable descending chain of finite type and stable semistar operations on $R$. Then the following are equivalent.
\begin{enumerate}
\item $*_{\mathcal{C}}$ is finite type.
\item
$\nu_{\mathcal{C}}$ has finite range.
\item
$\cap_{n\ge0}G(\nu,n)$ is a localizing system of $R$ of finite type.
\item
$\cap_{n\ge0}G(\nu,n)=G(\nu, m)$ for some $m \in \mathbb{N}$.
\end{enumerate}
\end{coro}

\begin{proof}
$(1) \Leftrightarrow (2)$: Note that $I \in G(\nu_{\mathcal{C}}, n)$ if and only if $\nu_{\mathcal{C}}(I)\ge n$ if and only if $I ^{*_{n}}=R^{*_{n}}$ if and only if $I \in \mathcal{F}^{*_{n}}$. Thus the conclusion follows from Lemma \ref{y1} and Lemma \ref{y2}. $(2) \Rightarrow (4)$: If the range of $\nu_{\mathcal{C}}$ is finite, then there is $m \geq 0$ with $G(\nu_{\mathcal{C}}, n)=G(\nu_{\mathcal{C}},m)$ for every $n\ge m$. Since $\{G(\nu_{\mathcal{C}},n)\}_{n\ge 0}$ is a descending chain of localizing systems, the claim holds. $(4) \Rightarrow (3)$: This follows from the fact that $G(\nu_{\mathcal{C}},n)$ is of finite type for each $n \ge 0$. $(1) \Leftrightarrow (3)$: Follows from Theorem \ref{vs}.
\end{proof}

\end{document}